\documentclass{compositio}

\usepackage[centertags]{amsmath}
\usepackage{amsfonts}
\usepackage{amssymb}
\usepackage{amsthm}
\usepackage{amscd}
\usepackage{chemarrow}

\renewcommand{\contentsname}{Contents}

\newcommand{\cO}{\mathcal{O}}

\newcommand{\h}{\mathfrak{h}}

\newcommand{\p}{\mathfrak{p}}

\newcommand{\ds}{\displaystyle}

\newcommand{\ra}{\rightarrow}
\newcommand{\xra}[1]{\xrightarrow{#1}}
\newcommand{\hra}{\hookrightarrow}

\newcommand{\eps}{\varepsilon}
\newcommand{\vphi}{\varphi}

\newcommand{\Zhat}{\widehat{\Z}}

\newcommand{\comment}[1]{}
\newcommand{\Q}{\mathbb{Q}}

\newcommand{\R}{\mathbb{R}}

\newcommand{\cD}{\mathcal{D}}
\newcommand{\cI}{\mathcal{I}}

\newcommand{\C}{\mathbb{C}}

\newcommand{\Qbar}{\overline{\Q}}

\newcommand{\T}{\mathbb{T}}

\newcommand{\Z}{\mathbb{Z}}
\newcommand{\F}{\mathbb{F}}

\newcommand{\Fbar}{\overline{\F}}

\newcommand{\isom}{\cong}
\newcommand{\n}{\mathfrak n}
\newcommand{\bc}{\mathbf{c}}
\newcommand{\Rhat}{\widehat{R}}
\newcommand{\Af}{\mathbb{A}_f}
\DeclareFontEncoding{OT2}{}{} 

\DeclareMathOperator{\Div}{Div}
\DeclareMathOperator{\SSS}{SS}
\DeclareMathOperator{\spn}{spn}
\DeclareMathOperator{\gen}{gen}
\DeclareMathOperator{\Tr}{Tr}

\DeclareMathOperator{\Pic}{Pic}
\DeclareMathOperator{\End}{End}

\DeclareMathOperator{\red}{red}

\DeclareMathOperator{\nr}{nr}

\DeclareMathOperator{\GL}{GL}

\DeclareMathOperator{\Gal}{Gal}
\DeclareMathOperator{\SL}{SL}

\DeclareMathOperator{\can}{can}
\DeclareMathOperator{\loc}{loc}
\DeclareMathOperator{\new}{new}
\DeclareMathOperator{\Eis}{Eis}
\DeclareMathOperator{\cusp}{cusp}
\DeclareMathOperator{\CM}{CM}
\DeclareMathOperator{\Res}{Res}
\DeclareMathOperator{\Sh}{Sh}
\DeclareMathOperator{\Stab}{Stab}
\DeclareMathOperator{\lcm}{lcm}
\DeclareMathOperator{\old}{old}
\DeclareMathOperator{\Hom}{Hom}

\theoremstyle{plain} 
\newtheorem{thm}{Theorem}[section] 
\newtheorem{prop}[thm]{Proposition}
\newtheorem{cor}[thm]{Corollary}
\newtheorem{lem}[thm]{Lemma}
\newtheorem{question}[thm]{Question}
\theoremstyle{definition} 
\theoremstyle{remark} 
\newtheorem{rem}{Remark}

\begin{abstract}
We prove new equidistribution results for Galois orbits of Heegner points 
with respect to reduction maps at inert primes. The arguments are based on 
two different techniques: primitive representations of integers by quadratic 
forms and distribution relations for Heegner points. Our results generalize 
one of the equidistribution theorems established by Cornut and Vatsal in the 
sense that we allow both the fundamental discriminant and the conductor to grow. 
Moreover, for fixed fundamental discriminant and variable conductor, we deduce 
an effective surjectivity theorem for the reduction map  from Heegner points 
to supersingular points at a fixed inert prime. Our results are applicable to 
the setting considered by Kolyvagin in the construction of the Heegner points 
Euler system.    
\end{abstract}

\begin{document}

\title[Equidistribution of Heegner Points]
{Equidistribution of Heegner Points and Ternary Quadratic Forms}
\author{Dimitar Jetchev}
\email{jetchev@gmail.com}
\address{IHES Le Bois-Marie\\35, route de Chartres\\ Bures-sur-Yvette\\France}

\author{Ben Kane}
\email{bkane@science.ru.nl}
\address{Department of Mathematics\\Radboud Universiteit\\Toernooiveld 1\\6525 Nijmegen\\Netherlands}

%
\thanks{}

\maketitle

\section{Introduction}\label{sec:intro}
Uniform distribution of Galois orbits of Heegner points with respect to reduction maps 
was the key step in the argument of Cornut and Vatsal for the proof of 
Mazur's conjecture on the non-triviality of Heegner points over the $p$-adic 
anticyclotomic tower (see~\cite{mazur:modcurves} for the statement; \cite{cornut:inventiones},~\cite{vatsal:uniform},~\cite{vatsal:special} and~\cite{cornut-vatsal} for the proofs). Both Cornut and Vatsal used ergodic theory techniques based on Ratner's 
theorem for unipotent flows on $p$-adic Lie groups (see~\cite{ratner:raghunatan}) in order to prove the results for simultaneous reduction maps (i.e., maps that reduce simultaneously $n$-tuples of Galois conjugates of Heegner points modulo a fixed inert prime $\ell$). Due to the $p$-adic nature of the ergodic techniques, one needs to fix the fundamental discriminant and vary the conductor $p$-adically.    

This paper proves a more general equidistribution result for single reduction maps, 
in the sense that both the fundamental discriminant and the conductor are allowed to vary 
and the only assumption on the conductor is that it is prime to the level of the modular curve. We avoid the ergodic theory by using arguments based on equidistribution 
of primitive representations of integers by quadratic forms in genera, as well as 
distribution relations of Heegner points and Hecke eigenvalue bounds. Along the way, 
we obtain a generalization of an equidistribution theorem for Gross points on definite 
Shimura curves established by Michel. Finally, we prove effective surjectivity results 
for sufficiently large Galois orbits with respect to reduction maps in the case 
when the fundamental discriminant is fixed and the conductor varies. 

\subsection{Notation and hypothesis}

Let $N \geq 1$ be an integer and let $X_0(N)_{/\Q}$ be the modular curve associated 
to the congruence subgroup $\Gamma_0(N)$ of $\SL_2(\Z)$. Let $\ell$ be a 
prime such that $(\ell, N) = 1$ and let $\cD_N$ be the set of all fundamental 
discriminants $D < 0$ such that every prime factor of $N$ is split in 
$K_D := \Q(\sqrt{D})$ and such that $\ell$ is inert in $K_D$. Let $\Omega_N$ be 
the set of all pairs $(D, c)$, where $D \in \cD_N$ and $(c, N) = 1$.  

Fix an embedding 
$\iota : \Qbar \hra \Qbar_\ell$ (i.e., a prime in $\Qbar$ lying 
above $\ell$). Let $(D, c) \in \Omega_N$ and let $\cO_{D, c}$ be the order of 
conductor $c$ in the quadratic imaginary field $K_D = \Q(\sqrt{D})$. Fix an ideal 
$\n_D \subset \cO_{D, 1}$ for which $\cO_{D, 1} / \n_D \isom \Z / N\Z$. 
For $(c, N) = 1$, $\n_{D, c} := \n_D \cap \cO_{D, c}$ is an invertible ideal of $\cO_{D, c}$. Consider the point $x_c = [\C/ \cO_{D, c} \ra \C / \n_{D, c}^{-1}] \in X_0(N)(\Qbar)$. 
By the theory of complex multiplication, it is defined over the ring class field $K_D[c]$ of conductor $c$ for $K_D$. We refer to that point as the higher Heegner point of conductor $c$. Let $\Gamma_{D, c} := \{\sigma x_c : \sigma \in \Gal(K_D[c]/K_D)\}$ be the corresponding Galois orbit. The fixed embedding $\iota$ gives us a prime in $K_D[c]$ 
above $\ell$. The choice of the embedding defines a reduction map 
$$
\red_\ell : X_0(N)(K_D[c]) \hookrightarrow X_0(N)(K_D[c]_\ell) = X_0(N)(\cO_{K_D[c]_\ell}) 
\xrightarrow{\mod \ell} X_0(N)(\Fbar_\ell) 
$$ 
where $\cO_{K_D[c]_\ell}$ is the ring of integers of the completion $K_D[c]_\ell$ (the 
equality in the middle follows from the valuative criterion of properness). Moreover, 
since $\ell$ is inert in $K_D$, then CM points for $K_D$ reduce to supersingular points modulo $\ell$ (see~\cite{deuring}). Let $X_0(N)^{\SSS}_{/\Fbar_\ell}$ be the set of 
supersingular points on $X_0(N)$ modulo $\ell$. It is well-known that these points are 
defined over $\F_{\ell^2}$. We will prove an equidistribution theorem according to which 
as $d_{c} := -Dc^2 \ra \infty$, every $s \in X_0(N)^{\SSS}_{/\F_{\ell^2}}$ will have the same number of preimages in $\Gamma_{D, c}$ under $\red_{\ell}$. We will state our result in terms of probability 
measures on the finite set $X_0(N)^{\SSS}_{/\F_{\ell^2}}$.  

\subsection{A canonical measure on $X_0(N)^{\SSS}_{/\Fbar_\ell}$}\label{subsec:measure}
Let $s \in X_0(N)^{\SSS}_{/\F_{\ell^2}}$ be a supersingular point. Then $s$ is represented by a pair $(\tilde{E}, \tilde{C})$ of a supersingular elliptic curve $\tilde{E}_{/\Fbar_\ell}$ and 
a cyclic subgroup $\tilde{C}$ of $\tilde{E}$ of order $N$. Following~\cite[\S 3]{ribet:modreps}, we refer to the pair $\mathbb{E} = (\tilde{E}, \tilde{C})$ as an 
\emph{enhanced elliptic curve} over $\F_{\ell^2}$. Homomorphisms of enhanced elliptic curves are defined in the obvious way. In particular, one could talk about endomorphisms and automorphisms of enhanced elliptic curves. 

Let $\mathbb{E} = (\tilde{E}, \tilde{C})$ be an enhanced elliptic curve representing the point $s$. The endomorphism algebra $\End(\tilde{E}) \otimes \Q$ is isomorphic to the unique quaternion algebra $B_{\ell, \infty}$ ramified precisely at $\ell$ and $\infty$. 
The endomorphism ring $\End(\tilde{E})$ is a maximal order 
in $B_{\ell, \infty}$ and the ring $\End(\mathbb{E})$ is an Eichler order of level $N$. 
Indeed, if $\lambda : \tilde{E} \ra \tilde{E} / \tilde{C}$ is the quotient map, then $\End(\tilde{E} / \tilde{C})$ can be viewed as a subring of $B_{\ell, \infty}$ via the 
map $\sigma \in \End(\tilde{E} / \tilde{C}) \mapsto \lambda^{-1} \sigma \lambda$. Then 
$\End(\mathbb{E})$ is the intersection of the two maximal orders $\End(\tilde{E})$ and 
$\End(\tilde{E} / \tilde{C})$. Let $R_s$ denote this Eichler order and let 
$w_s := \# R_s^\times$. We can use $w_s$ to define a canonical measure $\mu_{\can}$ 
on $X_0(N)^{\SSS}_{/\F_{\ell^2}}$ by 
$$
\mu_{\can}(s) := \frac{1/w_s}{ \sum_{s^\prime \in X_0(N)^{\SSS}_{/\F_{\ell^2}}} 1/w_{s'}}. 
$$   

\subsection{Main results}\label{subsec:main}

\noindent \emph{1. Equidistribution of Heegner points.} 
We can now state the main result of the paper. For $(D, c) \in \Omega_N$, define a measure 
$\mu_{D, c}$ on the finite set $X_0(N)^{\SSS}_{/\F_{\ell^2}}$ by  
$$
\mu_{D, c}(s) := \frac{\# \{ x \in \Gamma_{D, c} : \red_\ell(x) = s \}}{\# \Gamma_{D, c}},  
\quad s \in X_0(N)^{\SSS}_{/\F_{\ell^2}}. 
$$   

\begin{thm}\label{thm:main1}
The weak-* limit $\displaystyle \lim_{\substack{-Dc^2 \rightarrow \infty, \\ (D, c) \in \Omega_N}} \mu_{D, c}$ exists and equals $\mu_{\can}$.  
\end{thm}

\begin{rem}
To say that the weak-* limit of a sequence of measures $\{\mu_n\}$ on a finite set $X$ exists and converges to a measure $\mu$ on $X$ means that for each function $f : X \ra \R$, 
the limit $\ds \lim_{n \to \infty} \int_X f d\mu_n $ exists and equals $\ds \int_X f d\mu$. 
\end{rem}

\noindent \emph{2. Equidistribution of Gross points on the definite quaternion 
algebra $B_{\ell, \infty}$.} The curve $X_0(N)_{/\Q}$ can be viewed as a Shimura curve 
for the quaternion algebra $M_2(\Q)$, and thus, Heegner points can be regarded as 
CM points on the indefinite quaternion algebra $M_2(\Q)$. In the case of a totally definite 
quaternion algebra (e.g., $B_{\ell, \infty}$), the analogues of Heegner points (also known as Gross points) were studied in detail by Gross \cite{gross:heights}.

Let $G'$ be the algebraic group associated to $B_{\ell, \infty}^\times$ and let $I_1, \dots, I_h$ be left ideals representing the left ideal classes (corresponding to the double quotient $G'(\Q) \backslash G'(\mathbb{A}_f) / \Rhat^\times$). Let $R_1, \dots, R_h$ be the 
associated Eichler orders. Given a conductor $c$, the points of conductor $c$ are simply pairs $(f : \cO_c \hra {R_i}_{/R_i^\times}, R_i)$ of one of these orders $R_i$ and an $R_i^\times$-conjugacy classes of optimal embeddings $f : \cO_c \ra R_i$. 
Recall that $f : \cO_c \ra R$ is optimal if $f(K) \cap R = \cO_c$ 
(we have extended $f$ to an embedding $f : K \ra B_{\ell, \infty}$). 
Let $\tilde{\mu}_{D,c}([I_i])$ be the number of Gross points $(f, R_i)$ of conductor 
$c$ divided by the total number of Gross points of conductor 
$c$. Then $\tilde{mu}_{D, c}$ is a probability measure on $G'(\Q) \backslash G'(\mathbb{A}_f) / \Rhat^\times$. There is a canonical measure on $G'(\Q) \backslash G'(\mathbb{A}_f) / \Rhat^\times$ defined as 
$$
\tilde{\mu}_{\can}([I_k]) := \frac{1/w_k}{\sum_{i = 1}^h 1/w_i}.
$$ 
\begin{thm}\label{thm:main2}
The weak-* limit $\ds \lim_{\substack{-Dc^2 \rightarrow \infty, \\ (D, c) \in \Omega_N}} \tilde{\mu}_{D,c}$ exists and equals $\tilde{\mu}_{\can}$. 
\end{thm}

\begin{rem}
A similar statement (for trivial conductor $c=1$) has already been established by Michel \cite[Thm.3]{michel:subconvexity} using subconvexity bounds for $L$-functions and independently by Elkies, Ono, and Yang \cite[Theorem 1.2]{ElkiesOnoYang}.
\end{rem}

\begin{rem}
Both Theorem~\ref{thm:main1} and Theorem~\ref{thm:main2} hold in greater generality 
for CM points on indefinite and totally definite quaternion algebras, respectively, with 
respect to more general reduction maps at several primes. The more general statements 
will be the subject of a forthcoming paper. 
\end{rem}

\begin{rem}
We will see in Section~\ref{subsec:adelic_supersingular} that the canonical measures $\mu_{\can}$ and $\tilde{\mu}_{\can}$ indeed coincide. 
\end{rem}

\vspace{0.1in}

\noindent \emph{3. Congruences for Hilbert class polynomials under the $U$-operator.} Recall that for a function with Fourier expansion $\ds f(z)=\sum_{n\geq 0} a(n) q^n$ the operator $U(\ell)$ is defined by  $\ds f(z)|U(\ell) := \sum_{n\geq 0} a(\ell n) q^n$.  Elkies, Ono and Yang were interested in the equidistribution of Heegner points with respect to reduction maps which they used to study a certain congruence for the Hilbert class polynomial under the $U$-operator.  In particular, combining the case $N=1$ of Theorem \ref{thm:main1} with \cite[Thm~2.3 (1)]{ElkiesOnoYang} gives the following immediate corollary (the case $c=1$ is \cite[Thm.~1.1]{ElkiesOnoYang}):  
\begin{cor}
Let $H_{D,c}\in \Z[x]$ be the polynomial whose roots are precisely the $j$-invariants of those elliptic curves with CM by $\cO_{D,c}$.  Let $\ell$ be a prime which is nonsplit in $\cO_{D,c}$.  Then for $d_c = -Dc^2$ sufficiently large (depending on $\ell$) there exists a polynomial $P_{D,c,\ell}\in \Z[x]$ such that 
$$
H_{D,c}(j(z)) | U(\ell) \equiv P_{D,c,\ell}(j(z)) \pmod{\ell}.
$$
\end{cor}

\vspace{0.1in}

\noindent \emph{4. Effective surjectivity of $\red_\ell$.}
One consequence of both Theorem~\ref{thm:main1} and Theorem~\ref{thm:main2} is 
the fact that for sufficiently large discriminant $d_c = -Dc^2$, the reduction 
map from CM points of conductor $c$ to supersingular points is surjective. It is 
natural to ask whether this theorem can be made effective. The ineffectiveness of 
one of the ingredients used in our argument, Siegel's lower bound on the class number, prevents us from establishing an effective result when both $D$ and $c$ vary. 
Yet, fixing the fundamental discriminant $D$ and varying the conductor $c$, one 
can establish effective surjectivity theorems (see Theorem~\ref{thm:eff_surjectivity} and Lemma~\ref{lem:eff_onedim}).  



\section{Heegner points and optimal embeddings}\label{sec:deuring}
Let $s \in X_0(N)^{\SSS}_{/\F_{\ell^2}}$ be a supersingular point modulo $\ell$.
In this section we will establish a one-to-one correspondence between 
\[
\begin{array}{rcl}
\left \{
\begin{array}{c}
\textrm{Heegner points } x \textrm{ on } X_0(N) \textrm{ of conductor }c\\
\textrm{ reducing to } s \in X_0(N)^{\SSS}_{/\Fbar_\ell}
\end{array}
\right \} 
& 
\Longleftrightarrow
&
\left \{ 
\begin{array}{c}
R_s^\times-\textrm{conjugacy classes of} \\
\textrm{optimal embeddings } \cO_{D, c} \hra R_s
\end{array}
\right \}
\end{array}
\]
For $c = 1$, the above correspondence is known as Deuring lifting theorem (see \cite{deuring}) and has been subsequently refined (as a correspondence) by Gross and Zagier \cite[Prop.2.7]{gross-zagier:singular}. We will deduce the correspondence from a 
recent result of the first author and Cornut \cite{cornut-jetchev}. 

\subsection{Galois orbits of Heegner points}
We start by proving that there are exactly $2^{\nu(N)}$ Galois orbits of Heegner 
points of conductor $c$, where $\nu(N)$ is the number of distinct prime divisors 
of $N$. 

\begin{lem}\label{lem:galorbitsheeg}
Suppose that $(c, N) = 1$. Then there are exactly $2^{\nu(N)}$ Galois orbits of Heegner points of conductor $c$ on $X_0(N)$ and each of these orbits has size $\#\Pic(\cO_{D, c})$. 
\end{lem}

\begin{proof}
Consider the set of all Heegner points of conductor 
$c$ on $X_0(N)$. They could be described as pairs $([\mathfrak a], \mathfrak n)$ of an ideal class $[\mathfrak a]$ and an ideal 
$\mathfrak n \subset \cO_K$ with the property that $\cO_K / \mathfrak n \isom \Z / N\Z$. 
The last property is equivalent to the fact that $\mathfrak n$ is primitive of norm $N$ ($\mathfrak n$ 
being primitive means that there is no rational prime number dividing $\mathfrak n$). Equivalently, if $N = p_1^{e_1} \dots p_t^{e_t}$ are the distinct prime divisors 
of $N$, we want $\mathfrak n = \mathfrak p_1^{e_1} \dots \mathfrak p_t^{e_t}$, where $\mathfrak p_i$ is one of the primes of $\cO_K$ above $p_i$ (indeed, if both $\p_i$ and 
$\overline{\p}_i$ occur then $\n$ would be divisible by $p_i$ and hence, would not be primitive).  
\end{proof} 

\subsection{Modular curves and Shimura curves}
Let $\Gamma$ be a congruence subgroup of $\SL_2(\Z)$ and let $U = U(\Gamma)$ be the closure of $\Gamma$ in $\SL_2(\Af)$. 
The group $\SL_2(\Q)$ admits a left action on  
$\h$ by linear fractional transformations and a left action on $\SL_2(\mathbb{A}_f)$ by left multiplication. Thus, $\SL_2(\Q)$ acts on the left on $\h \times \SL_2(\mathbb{A}_f)$. Moreover, $U$ has a right action on $\h \times \SL_2(\mathbb{A}_f)$ by acting trivially on $\h$ and by right multiplication on $\SL_2(\mathbb{A}_f)$. Strong approximation (see~\cite[p.81]{vigneras:quaternion}) gives a homeomorphism  
$$
Y(\Gamma) := \Gamma \backslash \h \ra \SL_2(\Q) \backslash \h \times \SL_2(\mathbb{A}_f) / U, \ z \mapsto [z, 1].   
$$
Let $H$ be the compact open subgroup of $\GL_2(\Af)$ that is the closure 
(in $\GL_2(\Af)$) of the image of $U$ under the inclusion 
$\SL_2(\Af) \hra \GL_2(\Af)$. We define the Shimura curve corresponding to the 
compact open subgroup $H$ as 
$$
\Sh_H = \GL_2(\Q) \backslash (\C \backslash \R) \times \GL_2(\Af) / H. 
$$
We shall see that $\Sh_H$ is a disjoint union of two copies of $Y(\Gamma)$. Indeed, 
consider the map 
$$
\phi : \Sh_H \ra \Q^\times \backslash \{\pm 1\} \times \Af^\times / \det(H)
$$
given by $[z, g] \mapsto [\textrm{sgn}(\textrm{Im}(z)), \det(g)]$. The fiber of this 
map over the point $[+1, 1]$ is isomorphic to $\SL_2(\Q) \backslash \SL_2(\Af) / U \isom Y(\Gamma)$.  Since $\det(H)$ is open, it follows that the quotient $\Q^\times \backslash \Af^\times / \det(H)$ is discrete. Since $\Q^\times \backslash \Af^\times \isom \Zhat^\times$ is compact, the double quotient $\Q^\times \backslash \Af^\times / \det(H)$ 
is finite. The quotient $\Q^\times \backslash \{\pm 1\} \times \Af^\times / \det(H)$ 
describes the connected components of the Shimura curve $\Sh_H$. For instance, for 
classical modular curves,   
$$
\GL_2(\Q) \backslash (\C \backslash \R) \times \GL_2(\Af) / H \isom Y(\Gamma)^{+} \sqcup 
Y(\Gamma)^{-}. 
$$

\subsection{Adelic description of CM points}
\noindent \emph{1. CM points on the Shimura curve $\Sh_H$.} 
Fix an embedding $K \hra M_2(\Q)$. This gives us an embedding 
$T \hra \GL_2$, where $T := \Res_{K/\Q} K^\times$. Consider the set $\CM_H$ of 
all points of the form $[g, h] \in \Sh_H$ whose stabilizer is a torus isomorphic to 
$K^\times$. It is easy to verify that an element $z \in \C \backslash \R$ is in $K$ 
if and only if $\Stab_{\GL_2(\Q)}(z)$ is isomorphic to $\Res_{K/\Q} K^\times = T$. 
This allows us to conclude that $\CM_H$ admits an adelic description as the double 
quotient $T(\Q) \backslash \GL_2(\mathbb{A}_f) / H$. Indeed, a point in $\CM_H$ is represented by a pair $[z, g]$, where $z \in \C \backslash \R$ is in $K$ and $g \in \GL_2(\mathbb{A}_f)$. Since all $z \in K$ are $\GL_2(\Q)$-conjugates and since the stabilizer of each $z$ in $\GL_2(\Q)$ is isomorphic to $T(\Q)$, we obtain 
$$
\CM_H \isom T(\Q) \backslash \GL_2(\mathbb{A}_f) / H. 
$$

\vspace{0.1in}

\noindent \emph{3. Conductors of CM points.}
Here, we assume that $R = (R', R'')$ is an oriented Eichler order of $M_2(\Q)$ of 
level $N$ (i.e., $R'$ and $R''$ are maximal orders and $R = R' \cap R''$) and 
consider the Shimura curve $\Sh_H$, where $H = \widehat{R}^\times$. Consider 
the two degeneracy maps 
$$
\delta' : \CM_H \ra T(\Q) \backslash \GL_2(\mathbb{A}_f) / \widehat{R'}^\times
$$
and 
$$
\delta'' : \CM_H \ra T(\Q) \backslash \GL_2(\mathbb{A}_f) / \widehat{R''}^\times. 
$$

Given a CM point $x \in T(\Q) \backslash G(\Af) / \widehat{R}^\times$ such that $x = [g]$, let $x'$ and $x''$ be the images of $x$ in $T(\Q) \backslash G(\Af) / \widehat{R'}^\times$ and $T(\Q) \backslash G(\Af) / \widehat{R''}^\times$, respectively. The stabilizer
$$
\Stab_{\widehat{K}^\times}(x') = \widehat{K}^\times \cap g \widehat{R'}^\times g^{-1} = \widehat{\cO(x')}^\times
$$ 
for some order $\cO(x') \subseteq \cO_K$. Let $c(x')$ be the conductor of that order. Similarly, we obtain an integer $c(x'')$ for $R''$. 
The conductor $\bc(x)$ is then defined as
$$
\bc(x) := \lcm(c(x'), c(x'')). 
$$ 

\begin{rem}
Note that if $q$ is a prime that divides one of $c(x')$ and $c(x'')$, but not the other one, then $q$ necessarily divides $N$. This shows that if $(c, N) = 1$, all CM points 
of conductor $c$ will be in fact Heegner points (i.e, $c(x') = c(x'')$).    
\end{rem}

\begin{rem}
For $\Gamma = \Gamma_0(N)$, i.e., for the modular curve $X_0(N)$, these degeneracy maps correspond precisely to the two degeneracy maps $\delta_1, \delta_N : X_0(N) \ra X(1)$ that 
map $[E, C]$ to $[E]$ and $[E/C]$, respectively.  
\end{rem}

\begin{rem}
If $X = X_0(N)$, a CM point $[\tau] \in \Gamma_0(N) \backslash \h$ would correspond 
to the pair of $N$-isogenous CM elliptic curves $E' = \C / \langle 1, \tau \rangle$ and 
$E'' = \C / \langle 1, N\tau\rangle$. Then $\cO' = \End(E')$ and $\cO'' = \End(E'')$ are both orders in $K = \Q(\sqrt{-D})$. Let $c'$ and $c''$ be their conductors, respectively.  
The conductor of the point $[E', E'']$ is then $\bc([E', E'']) = \lcm(c', c'')$.   
\end{rem}

\subsection{Optimal embeddings and Gross points}\label{subsec:gross-points}
Let $B_{\ell, \infty}$ be the unique quaternion algebra ramified precisely at $\ell$ 
and $\infty$ and let $G' := B_{\ell, \infty}^\times$ be the corresponding algebraic 
group. Let $R_1, \dots, R_h$ be the Eichler orders of level $N$ defined in Section~\ref{subsec:main}. 

\begin{lem}
The set of pairs $(f: \cO \hra {R_i}_{/R_i^{\times}}, [R_i])$ of an ideal class $[R_i]$ 
of $G'(\Q) \backslash G'(\mathbb{A}_f) / \Rhat^\times$ and a 
$R_i^\times$-conjugacy class of optimal embeddings $f : \cO \hra {R_i}_{/R_i^\times}$ for some quadratic order $\cO$ in $K$ is in one-to-one correspondence with the double adelic quotient 
$$
T(\Q) \backslash G'(\mathbb{A}_f) / \widehat{R}^\times.  
$$  
\end{lem} 

\begin{proof}
Given an order $R_i$ representing an ideal class $[R_i]$, the set of 
$R_i^\times$-conjugacy classes of optimal embeddings $f : \cO \hra R_i$ is in 
bijection with $T(\Q) \backslash G'(\Q)$ (since all the embeddings of $K$ into 
$B_{\ell, \infty}$ are conjugate). Therefore, the set of the desired pairs is 
in bijection with 
$$
T(\Q) \backslash G'(\Q) \times G'(\Q) \backslash G'(\mathbb{A}_f) / \Rhat^\times \isom 
T(\Q) \backslash G'(\mathbb{A}_f) / \Rhat^\times. 
$$ 
\end{proof}

\subsection{Adelic description of supersingular points}\label{subsec:adelic_supersingular}
The set $X_0(N)^{\SSS}_{/\F_{\ell^2}}$ is in bijection with the double quotient 
$G'(\Q) \backslash G'(\Af) / \widehat{R'}^\times$ where $R'$ is an Eichler order of 
level $N$ for $B_{\ell, \infty}$ that is the ring of endomorphisms of a fixed enhanced supersingular elliptic curve $\mathbb{E}_0 = (\tilde{E}_0, \tilde{C}_0)$. We briefly summarize the bijection and refer the reader to~\cite[Prop.3.3]{ribet:modreps} for the details. 

Let $\mathbb{E}$ be any enhanced elliptic curve and take an endomorphism 
$\lambda \in \Hom(\mathbb{E}, \mathbb{E}_0) \otimes \Q$ (here, we use the fact that there is a 
single isogeny class of supersingular elliptic curves). One could use $\lambda$ to 
identify the adelic Tate module $\widehat{T}(\mathbb{E})$ with a sublattice of $\widehat{V}(\mathbb{E}_0)$. This means that there is a unique element $g \in G'(\mathbb{A}_f) / \widehat{R}^\times$ that sends this sublattice to $\widehat{T}(\mathbb{E}_0)$. Since $g$ is dependent on the choice of $\lambda$, 
it makes sense only in $G'(\Q) \backslash G'(\mathbb{A}_f) / \widehat{R}^\times$. 
This gives us a bijection 
$$
\vphi : X_0(N)^{\SSS}_{/\F_{\ell^2}} \ra G'(\Q) \backslash G'(\mathbb{A}_f) / \widehat{R}^\times.  
$$ 

\subsection{Heegner points on definite and indefinite quaternion algebras} 
The probability measures $\mu_c$ are defined in terms of the cardinalities 
$\left | \red_\ell^{-1}(s) \cap \Gamma_{D, c} \right |$. Let $s \in X_0(N)^{\SSS}_{/\F_{\ell^2}}$ be a supersingular point and  let $h(\cO_{D, c}, R_s)$ 
be the number of $R_s^\times$-conjugacy classes of optimal embeddings 
$\cO_{D, c} \hra R_s$.  

We will apply the theorem of Cornut--Jetchev \cite[Thm~1.5]{cornut-jetchev} 
together with the above adelic interpretations of CM points, optimal embeddings and supersingular points to deduce the following corollary:   

\begin{cor}
We have 
$$
h(\cO_{D, c}, R_s) = 2^{\nu(N)} \left | \{x \in \Gamma_{D, c} \colon \red_\ell(x) = s\} \right |. 
$$
\end{cor}

\begin{proof}
By~\cite[Thm~1.4]{cornut-jetchev} and the adelic interpretation of CM points on the definite and the indefinite algebras as well as the adelic desciption of the supersingular points, the subset of CM points on $X_0(N)$ of conductor $c$ reducing to a fixed supersingular point $s$ is in bijection with the $R_s^\times$-conjugacy classes of optimal embeddings $f : \cO_c \hra R_s$. Since $(c, N) = 1$, all CM points on $X_0(N)$ are Heegner points and by Lemma~\ref{lem:galorbitsheeg} there are exactly $2^{\nu(N)}$ such orbits.    
\end{proof}

The corollary shows that 
$$
\mu_c(s) = \frac{\left | \{ x \in \Gamma_{D, c} \colon \red_\ell(x) = s \} \right | }
{\left | \Gamma_{D, c}\right|} = \frac{h(\cO_{D, c}, R_s)}{\left | \Pic(\cO_{D, c})\right |}. 
$$ 

In section \ref{sec:proof}, the number $h(\cO_{D, c}, R_s)$ will be related to 
primitive representations of $d_c = -Dc^2$ by a certain quadratic form associated to $R_s$. 


\section{Modular forms of half-integral weight and Shimura correspondence} \label{sec:backmodular}
Let $\lambda$ be a non-negative integer and consider the space 
$M_{\lambda+\frac{1}{2}}(\Gamma_0(4M), \chi)$ of modular forms of weight $\ds \lambda + \frac{1}{2}$. Let $S_{\lambda + \frac{1}{2}}(\Gamma_0(4M), \chi)$ be the space of cusp forms.  Let $q:=e^{2\pi i z}$ and $\psi$ be an odd Dirichlet character of conductor $r(\psi)$.  We will refer to the form
\begin{equation}\label{eqn:one-dim}
h_{\psi, t}(z) := \sum_{m \geq 1} \psi(m) m e^{2 \pi i tm^2 z} = \sum_{m \geq 1} \psi(m) m q^{tm^2}\in S_{3/2}(4 r(\psi)^2,\psi\cdot \chi_{-4})
\end{equation}
as a \begin{it}one-dimensional theta series\end{it}.  Due to the exceptional behaviour of these forms, we will often decompose $S_{3/2}(4M)$ into the subspace spanned by one-dimensional theta series and the orthogonal complement of this space under the Petersson inner product, and then investigate each separately.

\subsection{Modular forms of half-integral weight and convolutions with $L$-series}
Suppose that $g(z) \in S_{\lambda + \frac{1}{2}}(\Gamma_0(4M), \chi)$. 
Let $t$ be a positive square-free integer and let 
$$
\ds \psi_t(n) := \chi(n) \left ( \frac{-1}{n} \right )^{\lambda} \left 
( \frac{t}{n}\right ).
$$
Suppose that the complex numbers $A_t(n)$ are defined by 
$$
\sum_{n = 1}^\infty \frac{A_t(n)}{n^{s}} := L(s - \lambda + 1, \psi_t) \cdot \sum_{n=1}^\infty \frac{b(tn^2)}{n^s}. 
$$ 
Shimura then proved that the $t$-th \begin{it}Shimura correspondence\end{it} $\ds S_{t, \lambda}(g(z)) := \sum_{n=1}^\infty A_t(n)q^n$ is 
a modular form in $M_{2\lambda}(\Gamma_0(2N), \chi^2)$ of weight $2\lambda$. 

Kohnen then defined a subspace $S_{\lambda + \frac{1}{2}}^+(\Gamma_0(4M))$, referred to as Kohnen's plus space, consisting of 
forms $g(z)$ of weight $\ds \lambda + \frac{1}{2}$ on $\Gamma_0(4M)$ with Fourier coefficients of the form 
$$
g(z) = \sum_{(-1)^\lambda n \equiv 0,1 \mod 4} b(n) q^n. 
$$
In this space Kohnen extended the definition of the Shimura correspondence $S_{t,\lambda}$ to $t':=(-1)^{\lambda}D$ where $D$ is a fundamental discriminant.  For $D\equiv 1\pmod{4}$ we take $S_{t',\lambda}:=S_{t,\lambda}$ as previously defined and for $D\equiv 0\pmod{4}$ we take $S_{t',\lambda}:=S_{t,\lambda}|U(4)$.  Kohnen's plus space decomposes into new and old subspaces as follows:  
$$
S_{\lambda + \frac{1}{2}}^+(\Gamma_0(4M)) = S_{\lambda + \frac{1}{2}}^{\new}(\Gamma_0(4M)) 
\oplus S_{\lambda + \frac{1}{2}}^{\old}(\Gamma_0(4M)). 
$$
Kohnen used this decomposition and the Shimura correspondences 
$$
S_{t', \lambda} : S_{\lambda + \frac{1}{2}}^{\new}(\Gamma_0(4M)) \ra S_{2\lambda}(\Gamma_0(N))
$$
to prove that there exists a finite linear combination of $S_{t', \lambda}$'s which provides 
an isomorphism  
\begin{equation}\label{eqn:correspondence}
S : S_{\lambda + \frac{1}{2}}^{\new}(\Gamma_0(4M)) \ra S_{2 \lambda}(\Gamma_0(N))  
\end{equation}
that is Hecke equivariant. The image of a half-integral weight Kohnen newform 
in $S_{\lambda + \frac{1}{2}}^{\new}(\Gamma_0(4M))$ is a newform in $S_{2\lambda}^{\new}(\Gamma_0(M))$ whose Hecke eigenvalues are the same.

\comment{
\subsection{Open questions}
\begin{question}
Is the Kohnen's plus space what we think the new space in 
$S_{\lambda+1/2}(\Gamma_0(4N))$ is, namely, we take the space generated by the images of $S_{\lambda+1/2}(\Gamma_0(N'))$, where $N' \mid N$ under all possible degeneracy maps 
and then declare $S^+_{\lambda+1/2}(\Gamma_0(4N))$ to be the orthogonal complement of 
that subspace under the Petersson inner product? 
\end{question}

\begin{question}
Is there a Shimura lift of $\theta_{Q_s} - \theta_{\gen(Q_s)}$? Should we use $S_D$ for 
the Shimura lift? 
\end{question}

\begin{itemize}
\item Space of modular forms of half-integral weight. 

\item Say what $S^+_{1/2+\lambda}(\Gamma_0(4M))$ is (write congruences)? 

\item If $f \in S_{1/2+\lambda}(\Gamma_0(4M))$ then if 
$$
\sum_{n \geq 1} \frac{A_t(n)}{n^s} = L(s - \lambda + 1) \sum_{n \geq 1} \frac{a_{tn^2}}{n^s},  
$$
then Shimura (see~\cite{blah}) defined the function 
$$
S_{t, \lambda}(f) = \sum_{n \geq 1} A_t(n) q^n,  
$$
which turns out to be a modular form in $M_{2\lambda}(4N)$. Kohnen 
(see~\cite{kohnen}) further defined a linear combination of the $S_{t, \lambda}$'s that give us a map
$$
S : S^{+}_{\lambda+1/2}(4M) \ra S_{\lambda}(\Gamma_0(M)). 
$$ 
\end{itemize}
}

\section{Equidistribution and ternary quadratic forms}\label{sec:proof}

In order to prove the main theorem, we establish the correspondence between 
optimal embeddings and primitive representations in Section~\ref{subsec:optimal-primrep} 
by associating a quadratic form $Q_s$ to the Eichler order $R_s$. We then compute 
the discriminant of that quadratic form. We introduce the theta series 
$\theta_{Q_s}$ associated to $Q_s$, as well as the series $\theta_{\gen(Q_s)}$ 
and $\theta_{\spn(Q_s)}$ associated to the genus $\gen(Q_s)$ and the spinor genus $\spn(Q_s)$ of $Q_s$, respectively. Finally, using that the form $\theta_{\gen(Q_s)} - \theta_{\spn(Q_s)}$ is in the space spanned by one-dimensional theta series, we are able to prove that the coefficients of $\gen(Q_s)$ and $\spn(Q_s)$ coincide away from the primes dividing 
$N\ell$. We use bounds on Fourier coefficients of modular forms of half-integral weight 
that lie in the the orthogonal complement (under the Petersson inner product) 
of the space spanned by one-dimensional theta series (due to Iwaniec and Duke) 
to conclude the proof of Theorem~\ref{thm:main1} and Theorem~\ref{thm:main2}.

\subsection{Optimal embeddings and primitive representations by ternary quadratic forms}\label{subsec:optimal-primrep}
Let $s \in X_0(N)^{\SSS}_{/\F_{\ell^2}}$ and let $R_s := \End(s)$ be the ring of endomorphisms 
of $s$. Recall the notation $w_s := \# R_s^\times$ and $u_{D, c} := \# \cO_{D, c}$. 

\vspace{0.1in}

\noindent \emph{1. A ternary quadratic form associated to an Eichler order.} 
Let $V \subset R_s$ be the set of elements of trace zero. Following 
\cite[pp.171--172]{gross:heights} define 
$$
G_s := (2R_s + \Z) \cap V. 
$$
The $\Z$-module $G_s$ is free of rank 3. Define a quadratic form $Q_s : G_s \ra \Q$ by  
$$
Q_s(b) := \nr(b).  
$$
 
\vspace{0.1in}

\noindent \emph{2. Correspondence between optimal embeddings and 
primitive representations.}

Let $f : \cO_{D, c} \hra R_s$ be an embedding (not necessarily optimal) and let 
$\beta := f(\sqrt{-d_c})$. Notice that $\Tr(\beta) = 0$ and $\nr(\beta) = d_c$. 
We claim that $\beta \in G_s$. Indeed, since $\ds \cO_{D, c} = \Z + \frac{d_c + \sqrt{-d_c}}{2}\Z$, it follows that 
$\ds 2f \left (\frac{d_c + \sqrt{-d_c}}{2} \right) = d_c + \beta$, i.e., 
$$
\beta \equiv -d_c \mod 2R_s. 
$$
Therefore, $\beta \in (\Z + 2R_s) \cap V = G_s$, i.e., $Q_s(\beta) = d_c$ is a representation. 

Conversely, suppose that $\beta \in G_s$ and $Q_s(\beta) = d_c$. We claim that 
$\beta \equiv -d_c \mod 2R_s$. Indeed, let $\beta = \gamma + 2r$ for some $\gamma \in \Z$ and 
$r \in R_s$. Then 
\begin{equation}\label{eqn:disc}
d_c = Q_s(\beta) = \nr(\beta) = \beta \overline{\beta} = -\beta^2 = -(\gamma + 2r)^2 \equiv -\gamma^2 \mod 4R_s. 
\end{equation}
Thus, 
$$
\beta = \gamma + 2r \equiv (\gamma + \gamma^2) - \gamma^2 \equiv d_c\equiv -d_c \mod 2R_s.   
$$
Now, we can define an embedding $f : \cO_{D, c} \hra R_s$ by 
$$
f \left ( \frac{d_c + \sqrt{-d_c}}{2} \right ) := \frac{d_c + \beta}{2} \in R_s.
$$

We next show under the established correspondence that optimal embeddings correspond to primitive representations. 

\begin{lem}
The embedding $f$ is optimal if and only if the representation $Q_s(\beta) = d_c$ is primitive. 
\end{lem}

\begin{proof}
Suppose that the representation $Q_s(\beta) = d_c$ is non-primitive. We will 
show that $f$ is not an optimal embedding. Indeed, let $\beta = k \alpha$ for some 
$k \in \Z$ and $\alpha \in G_s$. Then $\ds \nr(\alpha) = \frac{d_c}{k^2}$. Let 
$\ds d = \frac{d_c}{k^2}$. Consider the element $\ds \gamma = \frac{d + \alpha}{2}$. 
We claim that $\gamma \in f(K_D) \cap R_s$, but $\gamma \notin f(\cO_{D, c})$ which would imply that $f$ is a non-optimal embedding. 
First, $\ds \gamma = \frac{1}{k^2} f\left ( \frac{d_c + \sqrt{d_c}}{2}\right ) \in f(\cO_{D, c}) \otimes \Q$. Let $\ds \alpha = a + 2r$ for $a \in \Z$ and $r \in R_s$. 
Then $d = \nr(\alpha) = -\alpha^2 \equiv -a^2\mod 2R_s$. Thus, $\alpha = a + 2r \equiv -a^2 \equiv d\equiv -d \mod 2R_s$, i.e., $\gamma \in R_s$. Next, we show that $\gamma \notin f(\cO_{D, c})$. If $k \ne 2$ then $\ds \gamma = \frac{d + \alpha}{2}= \frac{d_c + \beta + dk-dk^2}{2k}$. Since $d_c+\beta = 2f(w) \notin kf(\cO_{D, c})$ then $\gamma \notin f(\cO_{D, c})$. If $k = 2$ then $\ds \gamma =\frac{d_c + \beta - 2d}{4}$. Since 
$d_c+\beta -2d = f(2w - 2d) \notin 4f(\cO_{D,c})$ we obtain the same statement. Thus, 
$\gamma \in f(K_D) \cap R_s$, but $\gamma \notin f(\cO_{D, c})$, i.e., the embedding is 
not optimal. 

Conversely, suppose that $f : \cO_{D, c} \hra R_s$ is a non-optimal embedding. 
Let $\cO = ( f(\cO_{D, c}) \otimes \Q ) \cap R_s$. It follows that $\cO \isom \cO_{D, c'}$, where $c = kc'$ for some $k > 1$. Now, we can choose $\alpha \in (2\cO+\Z) \cap V$, such that $Q_s(\alpha) = -Dc'^2$. Since $(\Z + 2\cO) \cap V$ is a free $\Z$-module of rank 1, we obtain $\beta = k \alpha$, i.e., the representation $Q_s(\beta) = -Dc^2$ is not primitive. 
This proves the lemma.  
\end{proof}

Thus, we have proved the following: 

\begin{prop}\label{prop:corresp}
There is a $\ds \frac{w_s}{ u_{D, c}}$-to-one correspondence between primitive representations 
of the integer $d_c = -Dc^2$ by $Q_s$ and optimal embeddings $f : \cO_{D, c} \hra R_s$.  
\end{prop}

\subsection{The discriminant of $Q_s$}\label{subsec:disc}
For what follows, we will need the discriminant of the quadratic form $Q_s$.
 
\begin{lem}\label{lem:disc}
The discriminant $D_{Q_s}$ of the quadratic form $Q_s$ is equal to $4N^2 \ell^2$.
\end{lem}

\begin{proof}
Let $p\neq \ell$ be a prime and let $v_p(N)=:n$.  Since $R_s$ is an Eichler order of level $N$, we know that $R_s\otimes \Z_p$ is an Eichler order of level $p^n$ of two-by-two matrices over $\Z_p$.  In particular (up to conjugation) we have
$$
R_s\otimes \Z_p = \left(\begin{array}{cc} \Z_p &\Z_p\\ p^n \Z_p & \Z_p\end{array}\right).
$$
Therefore, 
$$
G_s\otimes \Z_p =\left[  2 \left(\begin{array}{cc} \Z_p &\Z_p\\ p^n \Z_p & \Z_p\end{array}\right)+ \Z_p\right] \cap V = \left\{ \left(\begin{array}{cc} a & 2b\\ 2p^n c& -a\end{array}\right): a,b,c\in \Z_p \right\}.
$$
But then the local quadratic form $Q_{s, p}:=Q_s\otimes \Z_p$ is given by
$$
Q_{s, p}(a,b,c) =\left| \begin{array}{cc} a & 2b\\ 2p^n c& -a\end{array}\right| = -a^2 -4p^n bc
$$
The corresponding matrix for the quadratic form $Q_{s, p}$ is then 
\begin{equation}\label{eqn:Qpadic}
\left(\begin{array}{ccc} -1 &0&0 \\ 0& 0 & -2p^n\\ 0& -2p^n& 0\end{array}\right).
\end{equation}
The determinant of this matrix is $-4p^{2n}$.  Therefore, if $p\neq 2$ then $4$ is a unit and the contribution to the determinant of $Q_s$ is $p^{2n}$, while if $p=2$ the contribution is $4p^{2n}$.

Now consider the case $p=\ell\neq 2$.  In this case we have $R_s\otimes\Z_p$ is the unique maximal order of the unique division algebra, with $\Z_p$-basis $(1,\alpha,\beta,\gamma)$ satisfying $\alpha^2=-p$, $\beta^2=-1$ and $\gamma=\alpha\beta=-\beta\alpha$.  But then $G_s\otimes \Z_p$ has basis $(2\alpha,2\beta,2\gamma)$.  We obtain the quadratic form 
\begin{equation}\label{eqn:Qladic}
Q_{s, p}(2a\alpha+2b\beta+2c\gamma)= 4p a^2 + 4b^2 + 4pc^2,
\end{equation}
which is diagonal with discriminant $64 p^2$, contributing $p^2$ to the discriminant. 

For $p=\ell=2$ we note that since the Eichler order is locally isomorphic to the 
(unique) maximal order, Gross \cite[p. 177]{gross:heights} has shown that for $\alpha^2=\beta^2=\gamma^2=-1$ with $\gamma=\alpha\beta=-\beta\alpha$,
$$
G_s\otimes Z_p = \{ a\alpha +(a+2b)\beta+(a+2c)\gamma: a,b,c\in \Z_p\}.
$$
Thus the $p$-adic quadratic form is given by 
$$
Q_{s, p}(a,b,c) = -(3 a^2 + 4ab + 4ac +4b^2 +4c^2),
$$
with corresponding matrix 
\begin{equation}\label{eqn:Q2adic}
\left(\begin{array}{ccc} 3& 2 &2\\ 2&4& 0\\ 2&0&4\end{array}\right)
\end{equation}
The determinant of this matrix is $16$, and hence contributes $16=4\ell^2$ to the discriminant. 
\end{proof}

\subsection{The theta series associated to $Q_s$} 
Consider the theta series 
$$
\theta_{Q_s} := \sum_{\beta \in G_s} q^{Q_s(\beta)} = \sum_{d \geq 1} a_s(d) q^d. 
$$
Since $-Q_s(\beta) \equiv 0, 1 \mod 4$, we obtain that $a_s(d) \ne 0$ only if $-d$ 
is a discriminant, i.e., $-d \equiv 0, 1 \mod 4$. Thus, 
$$
\theta_{Q_s} = \sum_{\beta \in G_s} q^{Q_s(\beta)} = \sum_{-d \equiv 0, 1 \mod 4} a_s(d) q^d. 
$$
Recall the definition of Kohnen's plus space $M_{3/2}^{+}(\Gamma_0(4M))$ from section \ref{sec:backmodular}.

\begin{lem}\label{lem:thetalevel}
We have $\theta_{Q_s}\in M_{3/2}^{+}(\Gamma_0(4N\ell))$.
\end{lem}
\begin{proof}
Let $A$ be the matrix corresponding to $Q_s$.  It is well known that $\theta_{Q_s}\in M_{3/2}^+(\Gamma_0(4M))$, where $M$ is the minimal positive integer, such that $4M A^{-1}$ has coefficients that are even integers (see \cite[p. 39]{duke:ternary}). Since $A^{-1}$ has rational coefficients, it suffices to check that each coefficient of $4MA^{-1}$ has non-negative $p$-adic valuation for each $p$.  We then explicitly compute the inverse of equations (\ref{eqn:Qpadic}), (\ref{eqn:Qladic}) and (\ref{eqn:Q2adic}) to check that it has even integral coefficients when we multiply by $4p^{v_p(N\ell)}$.
\end{proof}

\subsection{The theta series associated to the genus and the spinor genus of $Q_s$} 
Let $Q$ be a ternary quadratic form. Let $\gen(Q)$ be the genus of $Q$ and let $\spn(Q)$ be the spinor genus of $Q$ (see~\cite[Ch.X]{o'meara} for the definitions). Let $\loc(Q)$ be the set of all integers $n$ that are everywhere locally represented by $Q$. Let 
$r_Q(n)$ (resp. $r^*_{Q}(n)$) be the number of representations (resp. primitive representations) of $n$ by $Q$. Let $w_Q$ be the number of automorphs of $Q$ 
(see~\cite{Jones1} for the definition).  

\vspace{0.1in}

\noindent \emph{1. Theta series associated to $\gen(Q)$.} Let 
\begin{equation}\label{eqn:gen}
r(\gen(Q), n) := \frac{\sum_{Q' \in \gen(Q)} r_{Q'}(n)/w_{Q'}}{\sum_{Q' \in \gen(Q)} 1/w_{Q'}}. 
\end{equation}
Similarly, define 
$$
r^*(\gen(Q), n) := \frac{\sum_{Q' \in \gen(Q)} r_{Q'}^*(n)/w_{Q'}}{\sum_{Q' \in \gen(Q)} 1/w_{Q'}}. 
$$
We define the theta series associated to $\gen(Q)$ as 
$$
\theta_{\gen(Q)} := \sum_{n \geq 1} r(\gen(Q), n) q^n. 
$$

\vspace{0.1in} 

By calculating local densities, Jones \cite[Thm.86]{Jones1} has shown that for $d_c=-Dc^2$
\begin{equation}\label{eqn:jones}
r^{*}(\gen(Q_s), d_c)=C\frac{h(-\Delta d_c)}{u_{\Delta D,c}}.  
\end{equation}
Here, $\Delta$ denotes the discriminant $D_{Q_s}$ of $Q_s$ divided by the square 
of the greatest common divisor of the determinants of all two-by-two minors of the matrix corresponding to $Q_s$, and $C$ only depends on the Legendre symbol $\ds \left(\frac{d_c}{D_{Q_s}}\right)$. One can calculate $\Delta$ $p$-adically using equations 
\eqref{eqn:Qpadic}, \eqref{eqn:Qladic}, and \eqref{eqn:Q2adic} to show that the greatest common divisor of the determinants of all two-by-two minors is precisely $\sqrt{D_{Q_s}}$.  Thus, $\Delta=1$.  

\vspace{0.1in}

\noindent \emph{2. Theta series associated to $\spn(Q)$.} We define the theta series associated to the spinor genus in a similar way. First, let  
\begin{equation}\label{eqn:spn}
r(\spn(Q), n) := \frac{\sum_{Q' \in \spn(Q)} r_{Q'}(n)/w_{Q'}}{\sum_{Q' \in \spn(Q)} 1/w_{Q'}}. 
\end{equation}
Similarly, let 
$$
r^*(\spn(Q), n) := \frac{\sum_{Q' \in \spn(Q)} r_{Q'}^*(n)/w_{Q'}}{\sum_{Q' \in \spn(Q)} 1/w_{Q'}}. 
$$
We also define 
$$
\theta_{\spn(Q)} := \sum_{n \geq 1} r(\spn(Q), n) q^n. 
$$

The theta series $\theta_{\gen(Q)}$ and $\theta_{\spn(Q)}$ are in the same space as $\theta_Q$ (by~\eqref{eqn:gen} and \eqref{eqn:spn} and the fact that $\theta_{Q'}$ are in the same space as $Q$ for all $Q' \in \gen(Q)$; see 
also~\cite[p.366]{hanke:local_densities}).

\subsection{Equidistribution in terms of quadratic forms}

In light of the correspondence obtained in Proposition~\ref{prop:corresp}, the required equidistribution results (Theorem~\ref{thm:main1} and Theorem~\ref{thm:main2}) are  equivalent to showing that 
\begin{equation}\label{eqn:quadmain}
\lim_{\substack{(D,c)\in \Omega_N\\ d_c\to\infty}}\frac{ r^*(Q_s, d_c) u_{D,c}}{2^{\nu(N)}\#\Gamma_{D,c}} = w_s\mu_{\can}(s).
\end{equation}
This result will be equivalent to showing that the limit
\begin{equation}\label{eqn:genlim}
f(s):=\lim_{\substack{(D,c)\in \Omega_N\\ d_c\to\infty}} \frac{ r^*(Q_s, d_c) u_{D,c}}{\# \Gamma_{D,c}}
\end{equation}
exists and is independent of the supersingular point $s$. Here, recall that $d_c := -Dc^2$.  

First, note that $\theta_{Q_s}-\theta_{\spn(Q_s)}$ is a modular form of weight $3/2$ that lies in the orthogonal complement of the space of one-dimensional theta series under the Petersson inner product \cite{schulze-pillot:thetareihen}.  Duke's bound for the Fourier coefficients of such forms \cite{duke:bound}, extending the work of Iwaniec \cite{iwaniec:bound} to forms of weight $3/2$, combined with M\"obius inversion, implies that 
$$
r^{*}(\spn(Q_s),d_c)-r^{*}(Q_s,d_c) = O(d_c^{\frac{13}{28}+\epsilon}).
$$
Siegel's lower bound for the class number \cite{siegel:class_number} (see also \cite[p. 149]{Cox1}) implies that $\#\Gamma_{D,c}\gg d_c^{\frac{1}{2}-\epsilon}$, so  
\begin{equation}\label{eqn:fspn}
\frac{ r^*(Q_s, d_c) u_{D,c}}{\# \Gamma_{D,c}} = \frac{ r^*(\spn(Q_s), d_c) u_{D,c}}{\# \Gamma_{D,c}} +  O(d_c^{-\frac{1}{28}+\epsilon})
\end{equation}

Thus, we only need to show independence and convergence of the limit for each spinor genus.  Since $r^*(\gen(Q_s),n)$ is independent of $s$ by definition, it will be natural to compare $r^*(\gen(Q_s),n)$ with $r^*(\spn(Q_s),n)$ in order to determine the desired independence.  

In particular, we have the following.
\begin{lem}\label{lem:rstargen}
The limit
$$
\lim_{k\to\infty} \frac{r^{*}(\gen(Q_s), -Dp^{2k})u_{D,p^k}}{\# \Gamma_{D,p^k}}
$$
exists and is independent of $p$ and $s$. 
\end{lem}
\begin{proof}
The independence on $s$ is clear from the definition of $\gen(Q_s)$.  We will apply \eqref{eqn:jones} to $n = -Dp^{2k}$.
First, recall (see \cite[Cor.7.28, p.148]{Cox1}) that for any discriminant $D_0 < 0$ and any prime $p$ we have
\begin{equation}\label{eqn:class_formula}
h(-D_0 p^{2k})= C p^k\left ( 1-\frac{1}{p} \left(\frac{-D_0}{p}\right)\right )\cdot \frac{u_{D_0,p^k}}{u_{D_0,1}} h(-D_0)
\end{equation}
Equation (\ref{eqn:class_formula}) with $D_0=D$ allows us to express $r^*(\gen(Q_s), -Dp^{2k})$ and $\#\Gamma_{D,p^k}$ as 
\begin{equation}\label{eqn:rstargen}
r^{*}(\gen(Q_s), -Dp^{2k}) = \frac{c_D h(-D p^{2k})}{u_{D, p^{k}}} = 
c_d Cp^k\left (1-\frac{1}{p}\left(\frac{-D}{p}\right)\right ) \frac{h(-D)}{u_{D,1}},
\end{equation}
\begin{equation}\label{eqn:gammasize}
\frac{\# \Gamma_{D, p^k}}{u_{D,p^k}} = p^k\left (1-\frac{1}{p}\left(\frac{-D}{p}\right)\right ) \frac{\# \Gamma_{D, 1}}{u_{D,1}}.
\end{equation}
Hence, for $k \geq 1$ we obtain
\begin{equation}\label{eqn:fprat}
\frac{r^{*}(\gen(Q_s), -Dp^{2k})u_{D,p^k}}{\# \Gamma_{D, p^k}} = c_D C \frac{h(-D)}{\# \Gamma_{D, 1}}.
\end{equation}
The result follows since the right-hand side of (\ref{eqn:fprat}) is independent of $k$ and $p$.
\end{proof}
We now define the restricted limit
\begin{equation}\label{eqn:primelim}
f_{D,p}(s):=\lim_{k\to\infty} \frac{ r^*(Q_s, -Dp^{2k}) u_{D,p^k}}{\# \Gamma_{D,p^k}}= \lim_{k\to\infty} \frac{ r^*(\spn(Q_s), -Dp^{2k}) u_{D,p^k}}{\# \Gamma_{D,p^k}}.
\end{equation}

The equidistribution result of Vatsal \cite[Thm.1.5]{vatsal:uniform} combined with Proposition \ref{prop:corresp} states the following:
\begin{lem}\label{lem:vatsal}[Vatsal]
For every $s\in X_0(N)^{\SSS}_{/\F_{\ell^2}}$, fundamental discriminant $D < 0$ and $p\nmid N\ell$ we have 
\begin{equation}\label{eqn:vatsal}
f_{D,p}(s)= w_s\mu_{\can}(s).
\end{equation}
\end{lem}
We will now use equation (\ref{eqn:vatsal}) to rewrite $r^*(\spn(Q_s),n)$ in terms of $r^*(\gen(Q_s),n)$ and then use equation (\ref{eqn:fspn}) to show that $f(s)$ exists and is independent of $s$.

\begin{prop}\label{prop:gen-spn}
Let $s \in X_0(N)^{\SSS}_{/\F_{\ell^2}}$ and $n = -Dc^2$, where $D <0$ is a fundamental discriminant.  Assume that $(c, N\ell) = 1$. Then 
$$
r^*(\gen(Q_s), n) = r^*(\spn(Q_s), n).
$$
\end{prop}
Let $a_m := r(\gen(Q_s), m) - r(\spn(Q_s), m)$. According to a result of Schulze-Pillot \cite{schulze-pillot:thetareihen} as well as Flicker \cite{flicker:automorphic}, Niwa \cite{niwa:theta}, Cipra and others \cite{cipra:niwa-shintani}), $\theta_{\gen(Q_s)} -\theta_{\spn(Q_s)}$ belongs to the subspace of cuspidal forms of weight 3/2 spanned by the one-dimensional theta series (see also~\cite{hanke:survey}).  Note that the Fourier coefficients of the one-dimensional theta series $h_{\psi, t}(z)$ defined in equation (\ref{eqn:one-dim}) vanish outside the square class $t\Z^2$.  Let 
\begin{equation}\label{eqn:lincomb}
\theta_{\gen(Q_s)} - \theta_{\spn(Q_s)} = \sum_{\psi, t} c_{\psi, t} h_{\psi, t}. 
\end{equation}

Let $h_{\psi, t}$ be one of the one-dimensional theta series in \eqref{eqn:lincomb} and let $4M=4N\ell$ be the level of $\theta_{\gen(Q_s)} - \theta_{\spn(Q_s)}$. The transformation law for modular forms of level $4M$ with Nebentypus $\chi$ implies that $\psi(mn)=\psi(n) \chi_t(m)$ for every $(m, M)=1$, where $\ds \chi_t(m)= \chi(m) \left(\frac{-t}{m}\right)$ (see~\cite[p.285]{schulze-pillot:thetareihen}). In addition, if $h_{\psi,t}\neq 0$ then $4t \mid M$ (see, e.g., \cite[Kor.2]{schulze-pillot:thetareihen}). 

\begin{lem}\label{lem:dfund}
Let $-d > 0$ be the smallest positive integer satisfying the following two conditions:
\begin{enumerate}
\item If $d = Dc^2$, where $D<0$ is a fundamental discriminant then $c$ is prime to $N\ell$;

\item $a_{-d} \ne 0$.    
\end{enumerate}
Then $c = 1$ and $d = D$ is a fundamental discriminant.
\end{lem}
\begin{proof}
It follows from \eqref{eqn:lincomb} and Lemma~\ref{lem:thetalevel} (since $(m,M)=1$) that  $\psi(m)=\chi_t(m)$. Hence, 
$$
a_{-d} = \sum_{-d = t m^2, \psi} c_{\psi, t} \psi(m)m = \sum_{-d = t m^2}\chi_t(m) m  \sum_{\psi}c_{\psi, t} \ne 0. 
$$
Hence there exists $t$ satisfying $\ds \sum_{\psi} c_{\psi, t} \ne 0$. Choose the minimal $t$ with this property and observe that 
$\ds a_t = \sum_{t', t=t'm^2} \sum_{\psi} c_{\psi, t'} \psi(m) m = \sum_{\psi} c_{\psi, t} \ne 0$. Hence, $t = -d$ and 
$$
a_{-d} = \sum_{-d = tm^2, \psi} c_{\psi, t} h_{\psi, t} = \sum_{\psi} c_{\psi, -d} h_{\psi,-d} \ne 0. 
$$
Now, if the conductor $c$ of $d$ were not equal to 1, it would have divided $M$ and hence would not have been prime to $N\ell$. Thus, the only possibility is that $c = 1$ and 
$d = D$ is a fundamental discriminant. 
\end{proof}

\begin{lem}\label{lem:coeff}
Let $D < 0$ be the fundamental discriminant from Lemma \ref{lem:dfund} and 
$(c, N\ell) = 1$. If $D_{Q_s}$ is the discriminant of the quadratic form $Q_s$ then 
$$
a_{-Dc^2} = c \left ( \frac{DD_{Q_s}}{c}\right ) a_{-D}. 
$$
In particular, for $c=p^k$ we have
$$
a_{-Dp^{2k}} = p^{k} \left ( \frac{-DD_{Q_s}}{p}\right )^k a_{-D}.
$$
\end{lem}
\begin{proof}
Note that $\ds a_{-Dc^2} = \sum_{-Dc^2=tm^2, \psi} c_{\psi, t} \psi(m) m$. We know that if $t = -D(c')^2$ for some $c' > 1$ then $h_{t, \psi} = 0$ 
(since $(c', M) = 1$). Hence, 
$$
a_{-Dc^2} = \sum_{\psi} c_{\psi, D} \psi(c) c = c \chi_{D}(c) \sum_{\psi} c_{\psi, D} = c \left ( \frac{DD_{Q_s}}{c} \right ) a_{-D}. 
$$
\end{proof}

\begin{proof}[Proof of Proposition \ref{prop:gen-spn}]
We will prove the statement by contradiction. Assume the contrary and let $n$ be the smallest integer whose square part is prime to $N\ell$ and such that 
$r(\gen(Q_s), n) \ne r(\spn(Q_s), n)$.  Lemma \ref{lem:dfund} implies that if $-n = Dc^2$ 
for a fundamental discriminant $D$ and a conductor $c$ then $c = 1$ and $n = -D$.  Let $p\nmid N\ell$ be a prime for which $\left(\frac{DD_{Q_s}}{p}\right)=-1$.  We will show that under these assumptions the limit $f_{-D,p}(s)$ does not exist, contradicting Lemma \ref{lem:vatsal}.  Using Lemma \ref{lem:coeff} and equation (\ref{eqn:gammasize}), we have 
\begin{eqnarray*}
f_{-D,p}(s) &= &\lim_{k\to\infty}\left(  \frac{\left( r^*(\spn(Q_s),-Dp^{2k})-r^*(\gen(Q_s), -Dp^{2k})\right) u_{D,p^k}}{\# \Gamma_{D,p^k}}+ \frac{r^*(\gen(Q_s), -Dp^{2k}) u_{D,p^k}}{\# \Gamma_{D,p^k}}\right)\\
&=& \lim_{k\to\infty} \left( - \frac{a_{-Dp^{2k}}}{\# \Gamma_{D,p^k}/u_{D,p^k}} +  \frac{r^*(\gen(Q_s), -Dp^{2k}) u_{D,p^k}}{\# \Gamma_{D,p^k}}\right)\\
&=& \lim_{k\to\infty}\left( - \frac{a_{-D} p^k \left(\frac{DD_{Q_s}}{p}\right)^k}{p^k\left(1-\frac{1}{p}\left(\frac{D}{p}\right)\right) \# \Gamma_{D,1}/u_{D,1}} + \frac{r^*(\gen(Q_s), -Dp^{2k}) u_{D,p^k}}{\# \Gamma_{D,p^k}}\right)\\
&=& \lim_{k\to\infty} \left( - \frac{a_{-D}  (-1)^k}{\left(1-\frac{1}{p}\left(\frac{D}{p}\right)\right) \# \Gamma_{D,1}/u_{D,1}} + \frac{r^*(\gen(Q_s), -Dp^{2k}) u_{D,p^k}}{\# \Gamma_{D,p^k}}\right).
\end{eqnarray*}
However, the limit 
$$
\lim_{k\to\infty} \frac{r^*(\gen(Q_s), -Dp^{2k}) u_{D,p^k}}{\# \Gamma_{D,p^k}}
$$
exists by Lemma \ref{lem:rstargen}.  Therefore, if the limit $f_{-D,p}(s)$ exists, then the limit 
$$
\lim_{k\to\infty} - \frac{a_{-D} (-1)^k}{\left(1-\frac{1}{p}\left(\frac{D}{p}\right)\right) \# \Gamma_{D,1}/u_{D,1}}
$$
must also exist.  But $a_{-D}\neq 0$ and the only dependence on $k$ is the term $(-1)^k$, leading to a contradiction.
\end{proof}

\subsection{Proof of the main theorem}
We are now ready to prove Theorem \ref{thm:main1} and Theorem~\ref{thm:main2}. Let $D < 0$ be a fundamental discriminant and let $c$ be an integer with $(c,N\ell)=1$.  Define 
$$
g_{D,c}(s):= \frac{ r^*(Q_s, -Dc^2) u_{D,c}}{\# \Gamma_{D,c}}
$$
and 
$$
h_{D,c}:= \frac{ r^*(\gen(Q_s), -Dc^2) u_{D,c}}{\# \Gamma_{D,c}}.
$$
Note that $h_{D,c}$ is independent of $s$.  Proposition \ref{prop:gen-spn} combined with equation (\ref{eqn:fspn}) gives
$$
g_{D,c}(s) = h_{D,c} + O_s\left((-Dc^2)^{-\frac{1}{28}+\epsilon}\right).
$$
We now divide by $w_{s}$ and sum over all $s'\in X_0(N)^{\SSS}_{/\F_{\ell^2}}$.  Recall from Proposition \ref{prop:corresp} that 
$$
\frac{g_{D,c}(s')\#\Gamma_{D,c}}{w_{s'}}=  r^*(Q_{s',-Dc^2})\frac{u_{D,c}}{w_{s'}}
$$
is the number of optimal embeddings of $ \cO_{D, c}$ into $R_{s'}$. Summing over all $s'\in X_0(N)^{\SSS}_{/\F_{\ell^2}}$ thus gives $\# \Gamma_{D,c}$.  Hence, we have
\begin{equation}
1= \sum_{s'\in X_0(N)^{\SSS}_{/\F_{\ell^2}}} \frac{g_{D,c}(s')}{w_{s'}} = h_{D,c} \sum_{s'\in X_0(N)^{\SSS}_{/\F_{\ell^2}}} \frac{1}{w_{s'}} + O\left((-Dc^2)^{-\frac{1}{28}+\epsilon}\right).
\end{equation}
Therefore 
$$
h_{D,c} = \frac{1}{\sum_{s'\in X_0(N)^{\SSS}_{/\F_{\ell^2}}}1/w_{s'}} + O\left((-Dc^2)^{-\frac{1}{28}+\epsilon}\right).
$$
Thus the limit $\ds \lim_{-Dc^2\to\infty} h_{D,c}$ exists, and we obtain
$$
f(s) = \lim_{-Dc^2\to\infty}g_{D,c}(s) = \lim_{-Dc^2\to\infty}\left[ h_{D,c}(s) +O\left((Dc^2)^{-\frac{1}{28}+\epsilon}\right)\right] = \frac{1}{\sum_{s'\in X_0(N)^{\SSS}_{/\F_{\ell^2}}}} = w_s\mu_{\can}(s).
$$
But this is precisely equation (\ref{eqn:quadmain}), and hence we obtain Theorem \ref{thm:main1}.

\begin{rem}
Due to dependence on Siegel's lower bound for the class number, Theorem \ref{thm:main1} is ineffective.  However, if we fix a fundamental discriminant $D<0$ and only vary the conductor $c$, then this result becomes effective due to known growth of the class number in a fixed square class.  Moreover, in a fixed square class the $-D$-th Shimura correspondence implies that the difference
$$
a_c(s):= g_{D,c}(s) - h_{D,c}
$$
are coefficients of a weight 2 cusp form.  Using Deligne's optimal bound, the error term can be improved to $O(c^{-1/2+\epsilon})$. Therefore, the error can be written as 
$$
O((-D)^{-\frac{1}{28}+\epsilon} c^{-\frac{1}{2} +\epsilon}).
$$
\end{rem}

\section{Distribution relations method}
In this section, we establish equidistribution when the fundamental discriminant 
$D < 0$ is fixed and the conductor varies using an alternative argument based on 
the distribution relations for Heegner points and Hecke eigenvalue bounds. 

\subsection{An easier equidistribution theorem}\label{sec:easy_case}

Here, we only consider a special infinite set of conductors $c$ and a fixed fundamental 
discriminant $D < 0$. Let $\mathcal P$ be the 
set of all primes $r \nmid N$, such that $r$ is inert in $K$. Let $\mathcal I$ be the 
set of all integers that are square-free products of primes in $\mathcal P$. Note that 
$\Lambda \subset \cI$. Under the same hypothesis as before, we will prove the following statement: 

\begin{thm}\label{thm:easier}
Given a Galois orbit $\Gamma_{D,c}$ let $\mu_{D,c}$ be the measure on $X_0(N)^{\SSS}_{/\F_{\ell^2}}$ defined as in Theorem~\ref{thm:main1}. Then 
$\displaystyle \lim_{\substack{c \rightarrow \infty, \\\ c \in \mathcal I}} \mu_{D,c} = \mu_{\can}$. 
\end{thm}

\begin{rem}
The assumption that $c \in \mathcal I$ is not necessary and the argument in the more 
general case is exactly the same, except for the more technical form of the 
distribution relations. Here, we prove only the less technical statement 
where the distribution relations are easier to work with (see Section~\ref{sec:distr}). 
\end{rem}

\subsection{Distribution relations}\label{sec:distr}

Let $X_c \in \Div(X_0(N))$ be defined as $\ds X_c := \sum_{\sigma \in \Gal(K[c]/K)} (x_c^\sigma)$. We will prove the following distribution relation: 

\begin{lem}\label{lem:distrel}
For any prime number $\ell$ which is inert in $K$ and any positive integer $c$ coprime 
to $\ell$, the following distribution relation holds: 
$$
X_{c\ell} = T_\ell X_c. 
$$  
\end{lem}

\begin{proof}
Let $S$ be a set of coset representatives for $\Gal(K[c\ell]/K[c]) / \Gal(K[c]/K)$. 
The distribution relation for Heegner points \cite[\S 6]{gross:heegner_points} 
is the following equality of divisors of degree $\ell+1$ on $X_0(N)$:  
$$
\Tr_{K[\ell c] / K[c]} (x_{c\ell}^\sigma) = T_\ell (x_c^\sigma), \ \sigma \in \Gal(K[c\ell]/K),   
$$
i.e., 
$$
\sum_{\tau \in \Gal(K[c\ell]/K[c])} (x_{c\ell}^{\sigma \tau}) = T_\ell (x_c^\sigma), \ \sigma 
\in \Gal(K[c\ell]/K). 
$$
Hence, 
$$
\sum_{\sigma \in S}\sum_{\tau \in \Gal(K[c\ell]/K[c])} (x_{c\ell}^{\sigma \tau}) = \sum_{\sigma \in S} T_\ell (x_c^\sigma), \ \sigma 
\in \Gal(K[c\ell]/K),  
$$
which implies 
$$
X_{c\ell} = T_\ell X_c. 
$$
\end{proof}

\subsection{Proof of the main theorem}

\begin{proof}[Proof of Theorem~\ref{thm:easier}]
First, we note that the reduction map $\red_\ell : X_0(N)_{/\Q} \ra X_0(N)_{/\F_\ell}$ 
defined in Section~\ref{sec:intro} is Hecke equivariant. Thus, 
$$
\red_\ell(X_{cr}) = T_r \red_\ell(X_c). 
$$ 
Next, $\red_\ell(X_{cr})$ and $\red_\ell(X_c)$ belong to the subgroup 
$\Div^{\SSS}(X_0(N)_{/\Fbar_\ell})$ of divisors supported on the supersingular 
points of $X_0(N)_{/\Fbar_\ell}$. The Hecke algebra $\T_{N\ell}$ acts on the vector space \linebreak $V_{\SSS} = \Div^{\SSS}(X_0(N)_{/\Fbar_\ell}) \otimes \Qbar$ via its $\ell$-new quotient 
$\T^{\ell-\new}_{N\ell}$ (see~\cite{serre:letter} or~\cite{parent}). Let 
$$
V_{\SSS} = V_{\Eis} \oplus \left ( \bigoplus_{f} V_f \right ) 
$$ 
be the eigenspace decomposition of $V$, where $f$ ranges over all normalized
eigenforms $f \in S_2^{\ell-\new}(\Gamma_0(N\ell))$, 
$$
V_f = \{v \in V_{\SSS}\ :\ T_r v = a_r(f) v \textrm{ for all primes } r \}, 
$$ 
and 
$$
V_{\Eis} = \{v \in V_{\SSS} \ : \ T_r v = (r+1)v \textrm{ for all primes } r\}. 
$$   
Here, $a_r(f)$ denotes the $r$-th Fourier coefficient of the eigenform $f$. 

Let $Y_c = \ds \frac{1}{\# \Pic(\cO_c)}\red_\ell(X_c) \in V_{\SSS}$. 
It is easy to see that $\ds Y_c = \sum_{s \in X_0(N)^{\SSS}_{/\F_{\ell^2}}} \mu_c(s) \cdot (s)$. 
We can write the decomposition of $Y_c$ as 
$$
Y_c = Y_{c, \Eis} + \sum_f Y_{c, f}, \ Y_{c, f} \in V_f, \ Y_{c, \Eis} \in V_{\Eis}.
$$
The distribution relation from Lemma~\ref{lem:distrel} implies that 
$\ds \# \Pic(\cO_{cr})Y_{cr} = \# \Pic(\cO_c) T_r Y_c$. Since 
$\# \Pic(\cO_{cr}) = (r+1)\# \Pic(\cO_c)$ then 
$$
\ds Y_{cr} = \frac{1}{r+1} T_r Y_c.
$$  
We use this equality to obtain 
$Y_{cr, \Eis} = Y_{c, \Eis}$ and 
$\ds Y_{cr, f} = \frac{a_r(f)}{r+1} Y_{c, f}$ for any normalized eigenform 
$f \in S_2^{\ell-\new}(\Gamma_0(N\ell))$. 

The Ramanujan-Petersson conjecture then implies that 
$$
\ds \frac{a_r(f)}{r+1} \leq \frac{2 r^{1/2}}{r+1} \leq \frac{2}{r^{1/2}}.
$$ 
Thus, we obtain by induction on the number of prime divisors of $c$ that   
$$
Y_{c} = Y_{1, \Eis} + O(c^{-1/2}). 
$$
This means that $\ds \lim_{\substack{c \ra \infty, \\\ c\in \cI}} Y_c = Y_{1, \Eis}$. 

Finally, one uses the result from Section~\ref{subsec:candiv} to conclude that 
$Y_{1, \Eis}$ is equal to the divisor associated to the canonical measure $\mu_{\can}$. 
\end{proof}

\subsection{The divisor $Y_{1, \Eis}$}\label{subsec:candiv}
Let  
$$
\ds D_{\mu_{\can}} := \sum_{s \in X_0(N)^{\SSS}_{/\F_{\ell^2}}} \mu_{\can}(s) \cdot (s) \in V_{\SSS}.
$$ 
It is well-known (see, e.g.,~\cite[Lem.2.5]{vatsal:uniform}) that the divisor 
$D_{\mu_{\can}}$ is Eisenstein. In other words,   
$$
T_r D_{\mu_{\can}} = (r+1)D_{\mu_{\can}}  
$$  
for every prime $(r, N) = 1$. Next, we verify that $D_{\mu_{\can}}$ is the same as the Eisenstein part $Y_{1, \Eis}$ of $Y_1$: 
\begin{lem}
We have 
$$
D_{\mu_{\can}} = Y_{1, \Eis}. 
$$
\end{lem}

\begin{proof}
First, note that $\deg(D_{\mu_{\can}}) = 1 = \deg(Y_1)$. Furthermore, a divisor 
$D \in \Div(X_0(N)^{\SSS}_{/\F_{\ell^2}}) \otimes \Qbar$ is cuspidal if and only if it has degree zero (see e.g.,~\cite{serre:letter}). Thus, $\deg(Y_{1, \cusp}) = 0$ and hence, $\deg(Y_{1, \Eis}) = 1$.  

Next, consider the exact sequence 
$$
0 \ra \Div^0(X_0(N)^{\SSS}) \ra \Div(X_0(N)^{\SSS}_{/\F_{\ell^2}}) \xra{\deg} \Z \ra 0, 
$$ 
and look at the divisor $D = D_{\mu_{\can}} - Y_{1, \Eis}$. We know that $\deg(D) = 0$ and hence, $D$ is cuspidal. At the same time, $D$ is Eisenstein. If $D \ne 0$ then one would obtain a contradiction by using the Hecke eigenvalue bounds for cusp forms. Thus, $D = 0$ and hence, $Y_{1, \Eis} = D_{\mu_{\can}}$. 
\end{proof}

\section{Effective surjectivity results}\label{sec:effective}
We have seen in Theorem \ref{thm:main1} that $\mu_{D,c} \to \mu_{\can}$ as $d_c := -Dc^2 \to \infty$.  In particular, for sufficiently large 
$d_c$ we have $\mu_{D,c}(s)>0$ for every $s\in X_0(N)^{\SSS}_{/\F_{\ell^2}}$, 
giving surjectivity of the reduction $\red_{\ell}$ from $\Gamma_{D,c}$ to $X_0(N)^{\SSS}_{/\F_{\ell^2}}$. Here, we discuss effective versions of this 
surjectivity result.  

Recall that the proof of Theorem~\ref{thm:main1} uses Siegel's lower bound on the class number (see \eqref{eqn:fspn}). Since Siegel's bound $\#\Gamma_{D,c}\gg_{c,\eps} D^{\frac{1}{2}-\eps}$ is ineffective due to the fact that Siegel proved this result by first assuming the truth of GRH for Dirichlet $L$-functions and then proved the bound 
again with a different implied constant depending on the location of a possible Siegel zero \cite{siegel:class_number}. The best known effective results are due to Oesterl\'{e} \cite{oesterle:class_number}, but the growth obtained is only logarithmic in $D$. Hence, the surjectivity will be ineffective whenever we allow the fundamental 
discriminant to vary. 

Thus, we fix a fundamental discriminant $D < 0$. Given a supersingular point 
$s \in X_0(N)^{\SSS}_{/\F_{\ell^2}}$, decompose $\theta_{Q_s}$ as
\begin{equation}\label{eqn:eff_decomp}
\theta_{Q_s} - \theta_{\spn(Q_s)} = \sum_{i=1}^r b_i g_i,
\end{equation}
where $b_i\in \C$ and $\{g_1, \dots, g_r\}$ is a fixed set of cuspidal Hecke 
eigenforms in the orthogonal complement (under the Petersson inner product) of 
the space spanned by one-dimensional theta series of weight $3/2$.  We will denote 
the $d$-th coefficient of $g_i$ by $a_{g_i}(d)$ and the $-D$-th Shimura correspondence (recall the extended definition in Section \ref{sec:backmodular} given by Kohnen for fundamental discriminants) by $G_i:=S_{-D, 1}(g_i)$.  Denote the number of distinct prime divisors of $c$ by $v(c)$.  

The following theorem establishes an effective bound for $c$ (depending on the 
decomposition \eqref{eqn:eff_decomp} and the fundamental discriminant $-D$) beyond which the preimage $\red_\ell^{-1}(s)$ is non-empty.  Taking the maximum occurring bound over all $s'\in X_0(N)^{\SSS}_{/\F_{\ell^2}}$ gives a bound depending only on $N$, $\ell$ and $D$ beyond which surjectivity must hold.

\begin{thm}\label{thm:eff_surjectivity}
Let $c > 2$ be an integer prime to $N\ell$ that satisfies the following inequality 
$$
\frac{c^{1/2}}{2^{2v(c)+1}\sigma_0(c)\log c} >  \frac{1}{\log 2} \frac{u_{D,1}}{\#\Gamma_{D,1}} 
\left ( \sum_{i=1}^r \left|b_i a_{g_i}(-D)\right| \right ) \left ( \sum_{s'\in X_0(N)^{\SSS}_{/\F_{\ell^2}}}1/w_{s'} \right ) 
$$
Then the reduction map $\red_\ell : \Gamma_{D, c} \ra X_0(N)_{/\F_{\ell^2}}^{\SSS}$ 
satisfies $\red_\ell^{-1}(s) \ne \varnothing$.  
\end{thm}

\begin{proof}
For the $\theta$-series $\theta_{Q_s}$ we have the decomposition 
$$
\theta_{Q_s}(z) = E(z) + H(z) + f(z), 
$$
where $E(z)$ is an Eisenstein series, $H(z)$ is in the space spanned by one-dimensional theta series of weight 3/2, and $f(z)$ is a cusp form in the orthogonal complement of the space spanned by one-dimensional theta series (see~\cite[p.156]{hanke:survey}). Moreover, from the work of Schulze-Pillot \cite{schulze-pillot:thetareihen}, we know that  
$\ds E(z) = \theta_{\gen(Q_s)}(z)$ and $\ds H(z) = \theta_{\spn(Q_s)}(z) - \theta_{\gen(Q_s)}(z)$. 

Let $a(n) := r(Q_s, n) - r(\spn(Q_s), n)$ be the $n$th Fourier coefficient of the 
form $f(z)$ and let $a^*(n) := r^*(Q_s, n) - r^*(\spn(Q_s), n)$.
Let $n > 0$ be an integer satisfying $(n, N\ell) = 1$. 
We know by Proposition~\ref{prop:gen-spn} that $r(\gen(Q_s), n) = r(\spn(Q_s), n)$. Therefore,  
$$
r(Q_s, n) = r(\gen(Q_s), n) + (r(Q_s, n) - r(\gen(Q_s), n)) = r(\gen(Q_s), n) + a(n). 
$$
Next, if $n = tc^2$ where $t$ is square-free, M\"{o}bius inversion gives us 
\begin{equation}
r^*(Q_s, n) = \sum_{c' \mid c} \mu(c') r(Q_s, n/c'^2) = r^*(\gen(Q_s), n) + \sum_{c' \mid c} \mu(c') a(n/c'^2).  
\end{equation}
By Proposition~\ref{prop:corresp}, we know that $s \in X_0(N)^{\SSS}_{/\F_{\ell^2}}$ is 
in the image of $\red_\ell : \Gamma_{D, c} \ra X_0(N)^{\SSS}_{/\F_{\ell^2}}$ if and only 
if $Q_s$ primitively represents $d_c = -Dc^2$. Thus, $s$ is not in the image of the reduction map if and only if $r^*(Q_s, d_c) = 0$, i.e., if and only if 
\begin{equation}\label{eqn:equiv}
r^*(\gen(Q_s), d_c) = - \sum_{c' \mid c} \mu(c') a(d_c/c'^2). 
\end{equation}
The left-hand side can be computed using Jones' formula and \cite[Cor.7.28,p.148]{Cox1} as it was applied previously for \eqref{eqn:rstargen}. We obtain
\begin{equation}\label{eqn:eff_class}
 r^*(\gen(Q_s), d_c) \sum_{s'\in X_0(N)^{\SSS}_{/\F_{\ell^2}}}1/w_{s'}=\frac{\#\Gamma_{D,c}}{u_{D,c}} = \left (\sum_{c'\mid c} \mu(c')\left(\frac{D}{c'}\right) \frac{c}{c'} \right ) \frac{\#\Gamma_{D,1}}{u_{D,1}}.
\end{equation}

For the right-hand side of \eqref{eqn:equiv}, we would like to express the Fourier coefficient $a(d_c)$ in terms of $a(-D)$. This cannot be done directly for an arbitrary cusp form $f$ in the orthogonal complement of the space of one-dimensional theta series, but could be achieved if $f$ were an eigenform (due to the recurrence relations of the Hecke operators). In order to get such a relation, we write  
$$
f(z) = \sum_{i = 1}^r b_i g_i(z), 
$$
where $g_i$'s are Hecke eigenforms of weight $3/2$ whose images $G_i$ under the $-D$-th Shimura correspondence $S_{-D, 1}$ are normalized Hecke eigenforms. 

Decomposing $\theta_{Q_s}-\theta_{\gen(Q_s)}$ gives
\begin{equation}\label{eqn:eff_astar}
a^*(d_c)= \sum_{i=1}^r b_i \sum_{c' \mid c}\mu(c') a_{g_i}\left(d_{c/c'}\right).
\end{equation}
If $g := g_i$ is a Hecke eigenform, the $-D$-th Shimura correspondence 
$G := S_{-D,1}(g) \in S_2(\Gamma_0(N\ell))$ is also a Hecke eigenform.  Assume further 
that $G$ is normalized so that $a_G(1)=1$.  By the multiplicity one theorem for forms of weight 2, there exists a newform $\widetilde{G}\in S_2(\Gamma_0(M))$ for some $M\mid N\ell$ 
such that $\ds G=\sum_{d \mid \frac{N\ell}{M}} C_{d} \widetilde{G} | V(d)$ for some constants $C_d$ (with $C_1=1$). Here, the operator $V(d)$ corresponds to one of the 
degeneracy maps (see e.g.,~\cite[p.28]{ono:book} for the definition). Notice that for $(c, N\ell) = 1$, the $c$th coefficient of 
$G$ corresponds to the $c$th coefficient of the newform $\widetilde{G}$.  
Since $c$ is relatively prime to the level, the $c$-th coefficient of $\widetilde{G}$ is determined by the eigenvalues under the Hecke operators.  

Using this connection and the definition of the $-D$-th Shimura correspondence to evaluate the coefficients of $\widetilde{G}$ (using the fact that $\widetilde{G}$ is normalized), the second author \cite[equation (4.2)]{Kane3} has shown for $c=p^m$ relatively prime to $F N\ell$,
$$
a_g(d_{cF})= a_g(d_F) \left(a_G(p^m)-\left(\frac{-D}{p}\right) a_G(p^{m-1})\right) = a_g(d_F) \sum_{c'|c}\mu(c')\left(\frac{-D}{c'}\right) a_G\left(\frac{c}{c'}\right).
$$
Here we have rewritten the right hand side so that extending by multiplicativity, it follows that
$$
a_g(d_c) = a_g(d_1) \sum_{c'\mid c}\mu(c')\left(\frac{D}{c'}\right) a_G\left(\frac{c}{c'}\right).
$$
Substituting this in equation (\ref{eqn:eff_astar}) gives the identity 
\begin{equation}\label{eqn:effc}
a^*(d_c) =\sum_{i=1}^r b_i a_{g_i}(d_1) \sum_{c'\mid c} \sum_{c''\mid \frac{c}{c'}} \mu(c')\mu(c'')\left(\frac{D}{c''}\right)a_{G_i}\left(\frac{c}{c'c''}\right).
\end{equation}
Thus we have established that the supersingular point $s\in X_0(N)^{\SSS}_{/\F_{\ell^2}}$ is not in the image of $\red_{\ell}$ from $\Gamma_{D,c}$ if and only if
\begin{equation}\label{eqn:eff_exp}
\frac{1}{\sum_{s'\in X_0(N)^{\SSS}_{/\F_{\ell^2}}}1/w_{s'}}\left ( \sum_{c'\mid c} \mu(c')\left(\frac{D}{c'}\right) \frac{c}{c'} \right ) \frac{\#\Gamma_{D,1}}{u_{D,1}} =  -\sum_{i=1}^r b_i a_{g_i}(d_1) \sum_{c'\mid c} \sum_{c''\mid \frac{c}{c'}} \mu(c')\mu(c'')\left(\frac{D}{c''}\right)a_{G_i}\left(\frac{c}{c'c''}\right).
\end{equation}
Now consider the Euler $\vphi$-function $\varphi(c):=\#\{m<c: (m,c)=1\}$.  Then
$$
\sum_{c'\mid c} \mu(c')\left(\frac{D}{c'}\right) \frac{c}{c'}\geq \varphi(c),
$$
since the inequality holds for $c$ being a prime power and both functions are multiplicative.  We can then use the explicit elementary bound $\ds \ds \varphi(c)\geq \frac{\log 2}{2} \frac{c}{\log c}$ for $c>2$ (cf. \cite[p.9]{handbook}).

We next pull the absolute value inside the sum on the right hand side of (\ref{eqn:eff_exp}) and use Deligne's optimal bound \cite{Deligne1} for integer weight cusp forms from the proof of the Weil conjectures, namely $|a_{G_i}(n)|\leq \sigma_0(n) n^{\frac{1}{2}}$.  Since $\# \{ c'\mid c: \mu(c')\neq 0\}=2^{v(c)}$ and $\sigma_0(c')\leq \sigma_0(c)$ for $c'\mid c$, we have
\begin{equation}\label{eqn:effc_bound}
|a^*(d_c)|\leq  2^{2v(c)} \sigma_0(c) c^{\frac{1}{2}} \sum_{i=1}^r |b_ia_{g_i}(D)|,
\end{equation}
giving the assertion. 
\end{proof}

\noindent \emph{2. The case $\#X_0(N)^{\SSS}_{/\F_{\ell^2}}$.} In the case when $\# X_0(N)^{\SSS}_{/\F_{\ell^2}}=2$ we obtain an explicit bound independent of $D$ beyond 
which surjectivity holds. Let $\ds m_s = \max\left(1,\frac{w_{s'}}{w_s}\right)$.  
 
\begin{lem}\label{lem:eff_onedim}
If $\# X_0(N)^{\SSS}_{/\F_{\ell^2}}=2$ then the inequality
\begin{equation}\label{eqn:eff_reduction}
\varphi(c)> m_s 2^{2v(c)} \sigma_0(c) c^{\frac{1}{2}}
\end{equation}
implies that the reduction $\red_{\ell}$ on $\Gamma_{D,c}$ is surjective for any fundamental discriminant $D<0$.
\end{lem}
\begin{proof}
Let $X_0(N)^{\SSS}_{/\F_{\ell^2}} = \{s, s'\}$. Recall that  
$$
\theta_{\gen(Q_s)} = \frac{\frac{1}{w_s} \theta_{Q_s} + \frac{1}{w_{s'}} \theta_{Q_{s'}}}{1/w_s + 1/w_{s'}}. 
$$
By Siegel's theorem (see~\cite[Thm.2(ii)]{duke-schulze-pillot}) there is a Hecke 
eigenform $g$ such that $\theta_{Q_s}=\theta_{\gen(Q_s)} + g$ and $\ds \theta_{Q_{s'}} = \theta_{\gen(Q_s)} -\frac{w_{s'}}{w_s} g$.  Since $r(Q_s,|D|)\geq 0$ and $r(Q_{s'},|D|)\geq 0$, we have $\ds |a_g(|D|)|\leq \max (1,\frac{w_{s'}}{w_s}) r(\gen(Q_s),|D|)$.  The lemma then follows immediately by combining equations (\ref{eqn:eff_class}) and (\ref{eqn:effc_bound}) with $b_1g_1 =g$ after canceling $r(\gen(Q_s),|D|)$ on both sides.
\end{proof}

Let $G = S_{-D, 1}(g)$ be the $-D$th Shimura correspondence of $g$ as defined in 
Section~\ref{sec:backmodular}. Define  
$$
r_c:=\frac{\sum_{c'\mid c} \mu(c') \left(\frac{-D}{c'}\right) \frac{c}{c'}}{\left| \sum_{c'\mid c} \mu(c') \sum_{c''\mid \frac{c}{c'}}\mu(c'') \left(\frac{-D}{c''}\right) a_G\left(\frac{c}{c' c''}\right) \right|},
$$
where we take $r_c=\infty$ by convention if the denominator is zero, and 
$$
\widetilde{r}_c:=\frac{\varphi(c)}{ 2^{2v(c)} \sigma_0(c) c^{\frac{1}{2}}}.
$$
By equations (\ref{eqn:eff_exp}) and (\ref{eqn:effc_bound}) if $r_c>m_s$ or $\widetilde{r}_c>m_s$ then $s$ is in the image of $\red_{\ell}$.
Note that both $r_c$ and $\widetilde{r}_c$ are multiplicative and 
$r_c\geq \widetilde{r}_c$. For $c=p^m$ we have 
$$
\widetilde{r}_c = \frac{p^{\frac{m}{2}-1}(p-1)}{4(m+1)}.
$$
For $p\geq 5$, $\widetilde{r}_c$ is increasing as a function of $m$, whereas for $p<5$ it is increasing for $m>2$.  For a constant $a$ and $m=1$ the inequality $\widetilde{r}_c>a$ is satisfied for 
$$
p>P_a:=\left(\frac{4a + \sqrt{16a^2 +4}}{2}\right)^2.
$$
For $p\leq P_a$ we use the fact that $\widetilde{r}_c$ is increasing exponentially as a function of $m$ to obtain a bound $M_{p,a}$ such that $m>M_{p,a}$ implies that $\widetilde{r}_c>a$.  Therefore, there are only finitely many choices for the pair $(p,m)$ with $m\geq 1$ for which $r_{p^m}\leq a$.  Let 
$$
C_a = \{(p, m) : r_{p^m} \leq a\}. 
$$ 
Computing $r_{p^m}$ explicitly for $m\leq M_{p,a}$ allows us to explicitly calculate 
$C_a$.

We first follow the above argument with $a=1$ to show that
$$
r_{\min} :=\prod_{p} \min_{m\geq 0} r_{p^m}
$$
is well defined and satisfies $r_c\geq r_{\min}$ for every $c$.  We will now use the above bounds with $\ds a:=\frac{m_s}{r_{\min}}$.  Let $c$ be an arbitrary integer such that $r_c\leq m_s$.  Write $c=p^mc'$ with $(p,c')=1$.  By multiplicativity we have 
$$
m_s\geq r_c=r_{c'}r_{p^m}\geq r_{\min} r_{p^m}.
$$
Therefore $r_{p^m}\leq a$, so $(p,m)\in C_a$, and it follows that 
$$
c\mid \prod_{(p,m)\in C_a} p^m.
$$
We can refine this argument by recursively computing
$$
S_v:=\{ c: v(c) =v, r_c\leq m_s \}.
$$
For $c'\in S_v$, consider 
$$
a':=a \frac{\prod_{(p,c')=1} \min_{m\geq 0} r_{p^m}}{r_{c'}}
$$
Then for $c=p^mc'$ with $(p,c)=1$, $r_c\in S_{v+1}$ if and only if $(p,m)\in C_{a'}$.  Constructing the resulting tree in this manner allows us to terminate the depth-first search when $C_{a'}$ is empty.

Proceeding in this manner, we obtain for $\ell=11$ and $N=1$ exactly 116 possible values of $c$ in the union of all $S_v$, the largest of which is 5124.  For $\ell=17$ and $N=1$ there are 93 possible values of $c$, the largest of which is 3990, and for $\ell=19$ and $N=1$ there are $165$ possible values of $c$, the largest of which is 8502.

\begin{acknowledgements}
We are grateful to Christophe Cornut for suggesting the problem and for the numerous discussions. We thank Barry Mazur, Philippe Michel, Steve Miller, Ken Ribet, William 
Stein and Tonghai Yang for helpful conversations. The first author thanks IHES, France 
for their kind hospitality and for providing a post-doctoral position during which a significant part of the research was completed. Part of the paper was written while the second author was in residence at IHES in France. He thanks the institute for providing a stimulating research environment.   
\end{acknowledgements}

\bibliographystyle{amsalpha}
\bibliography{biblio}

\end{document}